\documentclass[a4paper,12pt,final]{amsart}
\usepackage{times,a4wide,mathrsfs,amssymb,dsfont}

\newcommand{\C}{\mathbb{C}}
\newcommand{\Z}{\mathbb{Z}}

\newcommand{\QQ}{\mathbb{Q}}
\newcommand{\NN}{\mathbb{N}}
\newcommand{\PP}{\mathbb{P}}
\newcommand{\LL}{\mathbb{L}}
\newcommand{\A}{\mathbb{A}}

\newcommand{\XX}{\mathcal X}
\newcommand{\YY}{\mathcal Y}

\newcommand{\MM}{\mathcal M}

\newcommand{\ZZZ}{\mathcal Z}
\newcommand{\WWW}{\mathcal W}

\newcommand{\gr}{\hbox{Gr}}
\newcommand{\wt}{\widetilde}
\newcommand{\ima}{\hbox{Im}}
\newcommand{\rom}{\romannumeral}

\newtheorem{theorem}{Theorem}[section]
\newtheorem{claim}[theorem]{Claim}
\newtheorem{lemma}[theorem]{Lemma}

\newtheorem{corollary}[theorem]{Corollary}
\newtheorem{proposition}[theorem]{Proposition}

\newtheorem{nonumbering}{Theorem}

\newtheorem{convention}{Conventions}

\theoremstyle{definition}
\newtheorem{remark}[theorem]{Remark}
\newtheorem{definition}[theorem]{Definition}

\newtheorem{nonumberingt}{Acknowledgements}

\begin{document}
\author[Robert Laterveer]
{Robert Laterveer}

\address{Institut de Recherche Math\'ematique Avanc\'ee,
CNRS -- Universit\'e 
de Strasbourg,\
7 Rue Ren\'e Des\-car\-tes, 67084 Strasbourg CEDEX,
FRANCE.}
\email{robert.laterveer@math.unistra.fr}

\title[On the motive of intersections of two Grassmannians]{On the motive of intersections of two Grassmannians in $\PP^9$}

\begin{abstract} Using intersections of two Grassmannians in $\PP^9$, Ottem--Rennemo and Borisov--C\u{a}ld\u{a}raru--Perry have exhibited pairs of Calabi-Yau threefolds $X$ and $Y$ that are deformation equivalent, L--equivalent and derived equivalent, but not birational.
To complete the picture, we show that $X$ and $Y$ have isomorphic Chow motives.
\end{abstract}



\maketitle

\section{Introduction}

\noindent
Let $\hbox{Var}(\C)$ denote the category of algebraic varieties over the complex numbers $\C$, and let
 $K_0(\hbox{Var}(\C))$ denote the Grothendieck ring. This ring is a rather mysterious object. Its intricacy
is highlighted by Borisov \cite{Bor}, who showed that the class of the affine line $\LL$ is a zero--divisor in $K_0(\hbox{Var}(\C))$. Following on Borisov's pioneering result, recent years have seen a flurry of constructions of pairs of Calabi--Yau varieties 
 $X, Y$ that are {\em not\/} birational (and so $[X]\not=[Y]$ in the Grothendieck ring), but
  \[          ([X] -[Y]) \LL^r=0\ \ \ \hbox{in}\ K_0(\hbox{Var}(\C)) \ ,\]
  i.e., $X$ and $Y$ are ``L--equivalent'' in the sense of \cite{KS}.
  In most cases, the constructed varieties $X$ and $Y$ are also derived equivalent \cite{IMOU}, \cite{IMOU2}, \cite{Mar}, \cite{Kuz}, \cite{OR}, \cite{BCP}, \cite{KS}, \cite{HL}, \cite{Man}, \cite{KR}, \cite{KKM}.
  
   According to a conjecture made by Orlov \cite[Conjecture 1]{Or}, derived equivalent smooth projective varieties should have isomorphic Chow motives. This conjecture 
   is true for $K3$ surfaces \cite{Huy}, but is still widely
   open for Calabi--Yau varieties of dimension $\ge 3$. In \cite{L}, \cite{L2}, I verified Orlov's conjecture for the Calabi--Yau threefolds of Ito--Miura--Okawa--Ueda \cite{IMOU}, resp. the Calabi--Yau threefolds of Kapustka--Rampazzo \cite{KR}.
  The aim of the present note is to check that Orlov's conjecture is also true for the Calabi--Yau threefolds studied recently by Borisov--C\u{a}ld\u{a}raru--Perry \cite{BCP},   and independently by Ottem--Rennemo \cite{OR}.
  
  The threefolds of \cite{OR}, \cite{BCP} are called {\em GPK$^3$ threefolds}. The shorthand ``GPK$^3$'' stands for Gross--Popescu--Kanazawa--Kapustka--Kapustka,
   the authors of the papers \cite{GP}, \cite{K}, \cite{K2}, \cite{K3} where they first appeared (the shorthand ``GPK$^3$'' is coined in \cite{BCP}). These threefolds are constructed as follows. Given $W$ a $10$--dimensional vector space over $\C$, let $\PP:=\PP(W)$. Let $V$ be a $5$--dimensional vector space over $\C$, and choose isomorphisms 
    \[\phi_i\colon \wedge^2 V\to W\ , \ \ i=1,2\ .\] 
    Composing the Pl\"ucker embedding with the induced isomorphisms $\phi_i\colon\PP(\wedge^2 V)\cong\PP$, one obtains two embeddings of the Grassmannian $Gr(2,V)$ in $\PP$, whose images are denoted $Gr_i$, $i=1,2$. For $\phi_i$ generic, the intersection
   \[ X:= Gr_1\cap Gr_2\ \ \ \subset\ \PP \]
   is a smooth Calabi--Yau threefold, called a GPK$^3$ threefold. Let $Gr_i^\vee$ be the projective dual of $Gr_i$. The intersection
   \[ Y:=Gr_1^\vee\cap Gr_2^\vee\ \ \ \subset\ \PP^\vee \]  
   is again a smooth Calabi--Yau threefold, and it is deformation equivalent to $X$. The pair $X,Y$ are called {\em GPK$^3$ double mirrors\/}, and $X,Y$ are known to be Hodge equivalent, derived equivalent, L--equivalent, and in general not birational \cite{BCP}, \cite{OR}. In this note, we prove the following:

  \begin{nonumbering}[=theorem \ref{main}] Let $X, Y$ be two GPK$^3$ double mirrors. Then there is an isomorphism of Chow motives
   \[ h(X)\cong h(Y)\ \ \ \hbox{in}\ \MM_{{\rm rat}}\ .\]
   \end{nonumbering}
   
The proof of theorem \ref{main} is an elementary exercice in manipulating Chow groups and correspondences, based on a nice geometric relation between $X$ and $Y$ established in \cite{BCP} (cf. proposition \ref{bcp} below). The only ingredient in the proof that may perhaps not be completely standard is the use of Bloch's higher Chow groups (\cite{B2}, cf. also section \ref{shigh} below), and some results on higher Chow groups of piecewise trivial fibrations (section \ref{sfib} below).

\vskip0.6cm

\begin{convention} In this note, the word {\sl variety\/} will refer to a reduced irreducible scheme of finite type over the field of complex numbers $\C$. 
{\bf All Chow groups will be with $\QQ$--coefficients, unless indicated otherwise:} For a variety $X$, we will write $A_j(X):=CH_j(X)_{\QQ}$ for the Chow group of dimension $j$ cycles on $X$ with rational coefficients.
For $X$ smooth of dimension $n$, the notations $A_j(X)$ and $A^{n-j}(X)$ will be used interchangeably. 

The notations 
$A^j_{hom}(X)$ (and $A^j_{AJ}(X)$) will be used to indicate the subgroups of 
homologically trivial (resp. Abel--Jacobi trivial) cycles.
For a morphism between smooth varieties $f\colon X\to Y$, we will write $\Gamma_f\in A^\ast(X\times Y)$ for the graph of $f$, and ${}^t \Gamma_f\in A^\ast(Y\times X)$ for the transpose correspondence.

We will write  $\MM_{{\rm rat}}$ for the contravariant category of Chow motives (i.e., pure motives as in \cite{Sc}, \cite{MNP}, with Hom--groups defined using $A^\ast(X\times Y)_{}$).

We will write $H^j(X)=H^j(X,\QQ)$ for singular cohomology, and $H_j(X)=H_j^{BM}(X,\QQ)$ for Borel--Moore homology. 
\end{convention}

 \section{The Calabi--Yau threefolds}
 
 In this section we consider GPK$^3$ threefolds, as defined in the introduction.
 
\begin{proposition}[Ottem--Rennemo, Kanazawa \cite{OR}, \cite{K3}] The family of GPK$^3$ threefolds is locally complete. A GPK$^3$ threefold $X$ has Hodge numbers
  \[ h^{1,1}(X)=1\ ,\ \ \ h^{2,1}(X)=51\ .\]
  \end{proposition}
  
 \begin{proof} The first statement is \cite[Proposition 3.1]{OR}. The Hodge numbers are computed in \cite[Proposition 2.16]{K3}.
 \end{proof}

 \begin{theorem}[Ottem--Rennemo, Borisov--C\u{a}ld\u{a}raru-Perry \cite{OR}, \cite{BCP}]\label{or} Let $X,Y$ be a general pair of GPK$^3$ double mirrors. Then $X$ and $Y$ are not birational, and so
   \[ [X]\not= [Y]\ \ \ \hbox{in}\ K_0(\hbox{Var}(\C)) \ .\]
  
  However, one has
   \[  ([X] -[Y]) \LL^4=0\ \ \ \hbox{in}\ K_0(\hbox{Var}(\C)) \ .\]
   
   Moreover, $X$ and $Y$ are derived equivalent, i.e. there is an isomorphism of bounded derived categories
     \[ D^b(X)\cong D^b(Y)\ .\]
     
 In particular, there is an isomorphism of polarized Hodge structures
     \[ H^3(X,\Z)\ \cong\ H^3(Y,\Z)\ .\]
    \end{theorem}

     \begin{proof} Non--birationality is \cite[Theorem 1.2]{BCP}, and independently \cite[Theorem 4.1]{OR}. 
     Thanks to the birational invariance of the MRC--fibration, $X$ and $Y$ are not stably birational (cf. \cite[Proof of Theorem 2.12]{Bor}). The celebrated Larsen--Lunts result \cite{LL} implies that $[X]\not=[Y]$ in the Grothendieck ring.
     
     The L--equivalence is \cite[Theorem 1.6]{BCP}; it is a corollary of the geometric relation of proposition \ref{bcp} below. Derived equivalence is proven in \cite[Proposition 1.1]{OR},
     and also in \cite[Section 6.1]{KP}.
     
     The isomorphism of Hodge structures is a corollary of the derived equivalence, in view of
     \cite[Proposition 2.1 and Remark 2.3]{OR}. 
      \end{proof}

The argument of this note crucially relies on the following (for the notion of {\em piecewise trivial fibration\/}, cf. definition \ref{pt} below).

\begin{proposition}[Borisov--C\u{a}ld\u{a}raru--Perry \cite{BCP}]\label{bcp} Let $X,Y$ be a pair of GPK$^3$ double mirrors. There is a diagram
   \begin{equation}\label{diag} \begin{array}[c]{ccccccccc}   && F & \xrightarrow{\iota_F} & Q & \xleftarrow{} & G && \\
                &&&&&&&&\\
                   &{}^{\scriptstyle p_X} \swarrow \ \ && {}^{\scriptstyle p} \swarrow \ \ \ & & \ \ \ \searrow {}^{\scriptstyle q} & & \ \ \searrow {}^{\scriptstyle q_Y} & \\
                   &&&&&&&&\\
                   X & \xrightarrow{\iota} & Gr_1 &  &  & & Gr_2^\vee & \leftarrow & Y\\
                   \end{array}\end{equation}
                     Here, 
                     \[ Q:=  \sigma\times_{\PP\times\PP^\vee} (Gr_1\times Gr_2^\vee)  \]
 is the intersection of the natural incidence divisor $\sigma\subset\PP\times\PP^\vee$ with the product $Gr_1\times Gr_2^\vee\subset\PP\times\PP^\vee$,                   
                     the morphisms $p$ and $q$ are induced by the natural projections, the closed subvarieties $F,G$ are defined as $p^{-1}(X)$ resp. $q^{-1}(Y)$, and $p_X, q_Y$ are defined as the restrictions $p\vert_F$ resp. $q\vert_G$. The morphisms $p_X,q_Y$ are piecewise trivial fibrations with fibres $F_x$ resp. $G_y$ verifying
      \[ \begin{split} &A_i(F_x)_=A_i(G_y)_=\begin{cases} \QQ &\hbox{if\ }i=0,1,5\ ,\\
                                                                \QQ^2 &\hbox{if\ } i=2,3,4\ ,\\
                                                                \end{cases} \\
                                        &H_j(F_x)=H_j(G_y)=0\ \ \ \hbox{for\ all\ $j$\ odd}\ ,
                                   \end{split} \]                                                                                                           
                                   for all $x\in X, y\in Y$. However, over the open complements 
                   \[ U:=Gr_1\setminus X\ ,\ \ \ V:=Gr_2^\vee\setminus Y \ ,\]
                the restrictions $p_U:= p\vert_{p^{-1}(U)}, q_V:=q\vert_{q^{-1}(V)}$ are piecewise trivial fibrations with fibres $Q_u:=p^{-1}(u)$ resp. $Q_v:=q^{-1}(v)$ verifying
        \[    A_i(Q_u)_=A_i(Q_v)_=\begin{cases} \QQ &\hbox{if\ }i=0,1,4,5\ ,\\
                                                                \QQ^2 &\hbox{if\ } i=2,3\ ,\\
                                                                \end{cases} \]
                                                                for all $u\in U, v\in V$.         
\end{proposition}

\begin{proof} The diagram is constructed in \cite[Section 7]{BCP}. The computation of homology and Chow groups of the fibres of $p$ easily follows from the explicit description of the fibres as hyperplane sections of the Grassmannian $Gr(2,V)$ \cite[Section 7]{BCP}.
Precisely, as explained in \cite[Proof of Lemma 7.2]{BCP}, there exists a closed subvariety $Z\subset F_x$ such that $Z\cong\PP^2$, and the complement $C:=F_x\setminus Z$ is a Zariski locally trivial fibration over $\PP^2$, with fibres isomorphic to $\PP^3\setminus\PP^1$. Since neither $C$ nor $Z$ have odd--degree Borel--Moore homology, the same holds for $F_x$. As for even--degree homology, there is a commutative diagram with exact rows
  \[ \begin{array}[c]{cccccc}
     \to A_i(Z) &\to& A_i(F_x) &\to& A_i(C)&\to 0\\
     &&&&&\\
    \ \  \downarrow{\scriptstyle\cong}&&\downarrow&&\ \ \downarrow{\scriptstyle\cong}\\
     &&&&&\\
    0  \to H_{2i}(Z) &\to& H_{2i}(F_x) &\to& H_{2i}(C)&\ \ \to 0\ ,\\
    \end{array}\]
    where vertical arrows are cycle class maps. The left and right vertical arrow are isomorphisms, because of the above explicit description of $Z$ and $C$.
  This implies that the cycle class map induces isomorphisms $A_i(F_x)\cong H_{2i}(F_x)$ for all $i$. 

The bottom exact sequence of this diagram shows that
  \[ H_{2i}(F_x) = \begin{cases} H_{2i}(C) &\hbox{if\ }i=3,4,5\ ,\\
                                                   H_{2i}(C)\oplus \QQ &\hbox{if\ }i=0,1,2\ .\\
                                                   \end{cases}\]

Next, one remarks that the open $C$ (being a fibration over $\PP^2$ with fibre $T\cong \PP^3\setminus\PP^1$) has
  \[ H_{2i}(C)=\bigoplus_{\ell+m=2i} H_\ell(\PP^2)\otimes H_m(T)=\begin{cases} 0 &\hbox{if\ }i=0,1\ ,\\
                                                                                               \QQ &\hbox{if\ }i=2,5\ ,\\
                                                                                               \QQ^2 &\hbox{if\ }i=3,4\ .\\
                                                                                               \end{cases}\]
                                       Putting things together, this shows the statement for $H_{2i}(F_x)$.                                                        
                                                                                               
        (A more efficient, if less self--contained, way of determining the Betti numbers of $F_x$ is as follows. One has equality in the Grothendieck ring \cite[Lemma 7.2]{BCP}    
        \[ [ F_x]=  (\LL^2 +\LL+1)(\LL^3 +\LL^2+1)\ \ \ \hbox{in}\ K_0(\hbox{Var})\ .\]
        Let $W^\ast$ denote Deligne's weight filtration on Borel--Moore homology \cite{PS}.
       The ``virtual Betti number''
        \[ P_{2i}():= \sum_j (-1)^j \dim \gr_W^{-2i} H_j() \]
        is a functor on $K_0(\hbox{Var})$, and so 
         \[ P_{2i}(F_x)    = P_{2i}\Bigl(   (\LL^2 +\LL+1)(\LL^3 +\LL^2+1)\Bigr) = \begin{cases}     1 &\hbox{if\ }i=0,1,5\ ,\\
                                                                2 &\hbox{if\ } i=2,3,4\ .\\
                                                                \end{cases}  \]
         On the other hand, $F_x$ has no odd--degree homology, and the fact that $H_{2i}(F_x)$ is algebraic implies that $H_{2i}(F_x)$ is pure of weight $-2i$. It follows that $P_{2i}(F_x)=\dim H_{2i}(F_x)$.)

 The homology groups and Chow groups of the fibres $Q_u$ over $u\in U$ are determined similarly: according to loc. cit., there is a closed subvariety $Z\subset Q_u$ such that $Z$ is isomorphic to a smooth quadric in $\PP^4$, and the complement $Q_u\setminus W$ is a Zariski locally trivial fibration over $\PP^3$, with fibres isomorphic to $\PP^2\setminus\PP^1$.
  \end{proof}
  
  We record a lemma for later use:
  
  \begin{lemma}\label{Utriv} The open $U$ (and the open $V$) of proposition \ref{bcp} has trivial Chow groups, i.e. cycle class maps
    \[ A_i(U)\ \to\ H_{2i}(U)\ \]
    are injective.
    \end{lemma}
    
   \begin{proof} This is a standard argument. One has a commutative diagram with exact rows
     \[  \begin{array}[c]{cccccc}
     A_i(X)& \to&  A_i(Gr_1)& \to& A_i(U) & \to 0\\
      &&&&&\\
      \downarrow &&\ \ \downarrow{\cong}&& \downarrow{}&\\
      &&&&&\\
      H_{2i}(X)& \to&  H_{2i}(Gr_1)& \to& H_{2i}(U) & \to 0\\
      \end{array}\]
      (the middle vertical arrow is an isomorphism, as $Gr_1$ is a Grassmannian).
      Given $a\in A_i(U)$ homologically trivial, there exists $\bar{a}\in A_i(Gr_1)$ such that the homology class of $\bar{a}$ is supported on $X$. Using semisimplicity of polarized Hodge structures, the homology class of $\bar{a}$ is represented by a Hodge class in $H_{2i}(X)$. But $X$ being three--dimensional, the Hodge conjecture is known for $X$, and so $\bar{a}\in H_{2i}(Gr_1)$ is represented by a cycle $d\in A_i(X)$. The cycle $\bar{a}-d\in A_i(Gr_1)$ thus restricts to $a$ and is homologically trivial, hence rationally trivial.
     \end{proof}

\begin{remark}\label{sing} We observe in passing that the subvarieties $F,G$ in proposition \ref{bcp} must be singular. Indeed, the fibres $F_x, G_y$ have Picard number $1$, but the group of Weil divisors has dimension $2$, and so the fibres $F_x,G_y$ are not $\QQ$--factorial. By generic smoothness \cite[Corollary 10.7]{Ha}, it follows that $F,G$ cannot be smooth.
\end{remark}

\begin{remark} As explained in \cite[Section 5]{OR}, the $51$--dimensional family of GPK$^3$ threefolds degenerates to the $50$--dimensional family of Calabi--Yau threefolds first studied in \cite{K9}, \cite{IIM}. Generalized mirror pairs in this $50$--dimensional family are also derived equivalent and L--equivalent \cite{KR}, and have isomorphic Chow motives \cite{L2}.
\end{remark}

 \section{Higher Chow groups}
 \label{shigh}

 \begin{definition}[Bloch \cite{B2}, \cite{B3}] Let $\Delta^j\cong\A^j(\C)$ denote the standard $j$--simplex. For any quasi--projective variety $M$ and any $i\in\Z$, let $z_i^{simp}(M,\ast)$ denote the simplicial complex where $z_i(X,j)$ is the group of $(i+j)$--dimensional algebraic cycles in $M\times\Delta^j$ that meet the faces properly. Let $z_i^{}(M,\ast)$ denote the single complex associated to $z_i^{simp}(M,\ast)$. The higher Chow groups of $M$ are defined as
    \[ A_i(M,j):= H^j( z_i^{}(M,\ast)\otimes \QQ)\ .\]
 \end{definition}
 
 \begin{remark}\label{G} Clearly one has $A_i(M,0)\cong A_i(M)$. For a closed immersion, there is a long exact sequence of higher Chow groups \cite{B3}, \cite{Lev}, extending the usual ``localization exact sequence'' of Chow groups.
 Higher Chow groups are related to higher algebraic $K$--theory: there are isomorphisms
   \begin{equation}\label{KK} \gr_\gamma^{n-i} G_j(M)_\QQ\cong A_i(M,j)_{}  \ \ \ \hbox{for\ all\ }i,j \ ,\end{equation}
   where $G_j(M)$ is Quillen's higher $K$--theory group associated to the category of coherent sheaves on $M$, and $\gr^\ast_\gamma$ is the graded for the $\gamma$--filtration \cite{B2}. Higher Chow groups are also related to Voevodsky's motivic cohomology (defined as hypercohomology of a certain complex of Zariski sheaves) \cite{Fr}, \cite{MVW}.
  \end{remark}

  \section{Operational Chow cohomology}
  
  In what follows, we will rely on the existence of operational Chow cohomology, as constructed by Fulton--MacPherson. The precise definition does not matter here; we merely use the existence of a theory with good formal properties:
  
 \begin{theorem}[Fulton \cite{F}]\label{oper} There exists a contravariant functor
   \[ A^\ast()\colon\ \ \ \hbox{Var}_\C\ \to\ \hbox{Rings}\ \]
   (from the category of varieties with arbitrary morphisms to that of graded commutative rings),
   with the following properties: 
   
   \begin{itemize}
   
   \item for any $X$, and $b\in A^j(X)$ there is a cap--product 
    \[ b \cap ()\colon\ \ A_i(X)\ \to\ A_{i-j}(X)\ ,\]
    making $A_\ast(X)$ a graded $A^\ast(X)$--module;
    
    \item for $X$ smooth of dimension $n$, the map $A^j(X)\to A_{n-j}(X)$ given by
      \[  b\ \mapsto\ b \cap [X]\ \in A_{n-j}(X) \]
      is an isomorphism for all $j$;
    
   \item for any proper morphism $f\colon X\to Y$, there is a projection formula:
   \[    f_\ast (  f^\ast(b)\cap a)= b\cap f_\ast(a)\ \ \ \hbox{in}\ A_{i-j}(Y)\ \ \ \hbox{for\ any\ } b\in A^j(Y)\ ,\ a\in A_i(X)\ .  \]

   \end{itemize} 
     \end{theorem}
  
  \begin{proof} This is contained in \cite[Chapter 17]{F}.
    The projection formula is \cite[Section 17.3]{F}.
  \end{proof}

  \begin{remark} For quasi--projective varieties, there is another cohomology theory to pair with Chow groups: the assignment
    \[ CH^j(X):=\varinjlim A^j(Y)\ ,\]
    where the limit is over all smooth quasi--projective varieties $Y$ with a morphism to $X$. As shown in \cite{BFM}, \cite[Example 8.3.13]{F}, this theory satisfies the formal properties of theorem \ref{oper}. Since in this note, we are only interested in quasi--projective varieties, we might as well work with this theory rather than operational Chow cohomology.
 \end{remark}

 \section{Piecewise trivial fibrations}
 \label{sfib}
 
 This section contains two auxiliary results, propositions \ref{p2} and \ref{p1}. The first is about Chow groups of the open complement $R:=Q\setminus F$ of proposition \ref{bcp}; the second concerns the Chow groups of the singular variety $F$. 
   
 \begin{definition}[Section 4.2 in \cite{Seb}]\label{pt} Let $p\colon M\to N$ be a projective surjective morphism between quasi--projective varieties. We say that $p$ is a {\em piecewise trivial fibration with fibre $F$\/}
 if there is a finite partition $N=\cup_j N_j$, where $N_j\subset N$ is locally closed and there is an isomorphism of $N_j$--schemes $p^{-1}(N_j)\cong N_j\times F$ for all $j$.
  \end{definition}

 \begin{proposition}\label{p2} Let $U:=Gr_1\setminus X$ and  $R:= Q\setminus F$ and $p_U\colon R\to U$ be as in proposition \ref{bcp}. 
 
 \noindent
 (\rom1) Let $h\in A^1(R)$ be a hyperplane section, and let $h^j\colon A_i(R)\to A_{i-j}(R)$ denote the map induced by intersecting with $h^j$. There are isomorphisms
   \[ \begin{split} &\Phi_0\colon   \ \ \    A_{0}(U)\ \xrightarrow{\cong}\ A_0(R)\ ,\\ 
                       &\Phi_1 \colon\ \ \  A_1(U)\oplus A_{0}(U)\ \xrightarrow{\cong}\ A_1(R)\ ,\\ 
                        & \Phi_2\colon \ \ \    A_2(U)\oplus A_1(U)  \oplus   A_0(U)^{\oplus 2}    \ \xrightarrow{\cong}\ A_2(R)\  \ ,\\
                          &\Phi_3\colon \ \ \  A_3(U)^{}
                             \oplus A_2(U)^{} \oplus A_1(U)^{\oplus 2} \oplus A_0(U)^{\oplus 2} \ \xrightarrow{\cong}\ A_3(R)\  ,\\
                          &\Phi_4\colon \ \ \ A_3(U)\oplus A_2(U)^{\oplus 2}\oplus A_1(U)^{\oplus 2}\oplus A_0(U)\ \xrightarrow{\cong}\ A_4(R)\ ,\\
                           &\Phi_5\colon \ \ \ A_3(U)^{\oplus 2}\oplus A_2(U)^{\oplus 2}\oplus A_1(U)^{}\oplus A_0(U)\ \xrightarrow{\cong}\ A_5(R)\ .\\
                                                      \end{split}\]
              The maps are defined as
            \[  \begin{split}  
                 &\Phi_0:= h^5\circ (p_U)^\ast\ ,\\
                        &\Phi_1 := \Bigl(  h^5\circ (p_U)^\ast, h^4\circ (p_U)^\ast\Bigr) \ ,\\       
                            & \Phi_2:=\Bigl(   h^5\circ (p_U)^\ast, h^4\circ (p_U)^\ast, h^3\circ (p_U)^\ast, (b\cdot h)\circ (p_U)^\ast    \Bigr)\ ,\\
                                      & \Phi_3:=\Bigl(    h^5\circ (p_U)^\ast, h^4\circ (p_U)^\ast, h^3\circ (p_U)^\ast, (b\cdot h)\circ (p_U)^\ast,  h^2\circ (p_U)^\ast, 
                                      b\circ (p_U)^\ast           \Bigr)\ ,\\
                                       & \Phi_4:=\Bigl(  h^4\circ (p_U)^\ast, h^3\circ (p_U)^\ast, (b\cdot h)\circ (p_U)^\ast,  h^2\circ (p_U)^\ast, 
                                      b\circ (p_U)^\ast , h\circ (p_U)^\ast          \Bigr)\ ,\\    
                                       & \Phi_5:=\Bigl(   h^3\circ (p_U)^\ast, (b\cdot h)\circ (p_U)^\ast,  h^2\circ (p_U)^\ast, 
                                      b\circ (p_U)^\ast   , h\circ (p_U)^\ast, (p_U)^\ast        \Bigr)\ ,\\                                                                        
                                     \end{split} \]
                                     where $b\in A^2(R)$ is a class made explicit in the proof (and $b\colon A_i(R)\to A_{i-2}(R)$ denotes the operation of intersecting with $b$, and similarly for $b\cdot h$).
                               
  \noindent
 (\rom2) 
   \[ A_i^{hom}(R)=0\ \ \ \forall i\ .\]
  \end{proposition}
 
 \begin{proof} 
 
 \noindent
 (\rom1)
 As we have seen in proposition \ref{bcp}, the morphism $p_U\colon R\to U$ is a piecewise trivial fibration, with fibre $R_u$. Let $T$ denote the tautological bundle on the Grassmannian $Gr_2$, and define
  \[  b:=   \bigl( Gr_1\times  c_2(T)\bigr)\vert_R\ \ \ \in\ A^2(R)\ .\]
  The ($5$--dimensional) fibres $R_u$ of the fibration $p_U\colon R\to U$ verify
   \[ A^i(R_u)=  \begin{cases}     \QQ\cdot h^i\vert_{R_u}  &\ \hbox{if}\ i=0,1,4,5\ ,\\
                                       \QQ\cdot h^2\vert_{R_u} \oplus \QQ\cdot b\vert_{R_u} &\ \hbox{if}\ i=2\ ,\\
                             \QQ\cdot h^3\vert_{R_u} \oplus \QQ\cdot (b\cdot h)\vert_{R_u} &\ \hbox{if}\ i=3\ .\\
                         \end{cases}\]

To prove the isomorphisms of Chow groups of (\rom1), it is more convenient to prove a more general statement for {\em higher Chow groups\/}. That is, we consider
maps
  \[  \Phi_i^j\colon\ \ \ \bigoplus A_{i-k}^{}(U,j)\ \xrightarrow{}\ A_i^{}(R,j)\ \]
  such that $\Phi_i^0=\Phi_i$ (the maps $\Phi_i^j$ are defined just as the $\Phi_i$, using $(p_U)^\ast$ and intersecting with $h$ and $b$). We now claim that the maps $\Phi_i^j$ are isomorphisms for all $i=0,\ldots,5$ and all $j$. The $j=0$ case of this claim proves (\rom1).
  
 To prove the claim, we exploit the piecewise triviality of the fibration $p_U$. 
 Up to subdividing some more, we may suppose the strata $U_k$ (and hence also the strata $R_k$) are {\em smooth\/}.
 We will use the notation
   \[ U_{\le s}:= \bigcup_{i\le s} U_i\ ,\ \ \ R_{\le s}:=\bigcup_{i\le k} R_i\ .\]
  For any $s$, the morphism
   $p_{U_{\le s}}\colon R_{\le s}\to U_{\le s}$ is a piecewise trivial fibration (with fibre $R_u$). For $s$ large enough, $R_{\le s}=R$. Since $U$ and $R$ are smooth, we may suppose the $U_s$ are ordered in such a way that the $U_{\le s}$ (and hence the $R_{\le s}$) are {\em smooth\/}.
   
     The morphism $p_U$ is flat of relative dimension $5$, and so there is a commutative diagram of complexes (where rows are exact triangles)
      \[ \begin{array}[c]{cccccc}
            z_{i+5}(R_{\le s-1},\ast) &\to& z_{i+5}(R_{\le s},\ast) &\to& z_{i+5}(R_s,\ast) &\to\\
            &&&&&\\
            \uparrow{\scriptstyle (p_{U_{\le s-1}})^\ast}&& \uparrow{\scriptstyle p_{U_{\le s}}^\ast} &&       \uparrow{\scriptstyle (p_{U_s})^\ast}  \\
            &&&&&\\
              z_{i}(U_{\le s-1},\ast) &\to& z_{i}(U_{\le s},\ast) &\to& z_{i}(U_s,\ast) &\to\\
              \end{array}\]
        Also, given a codimension $\ell$ subvariety $M\subset R$, let $z_i^M(R_{\le k},\ast)\subset z_i(R_{\le k},\ast)$ denote the subcomplex formed by cycles in general position with respect to $M$. The inclusion $z_i^M(R_{\le k},\ast)\subset z_i(R_{\le k},\ast)$ is a quasi--isomorphism \cite[Lemma 4.2]{B2}.
          The projection formula for higher Chow groups \cite[Exercice 5.8(\rom1)]{B2} gives a commutative diagram up to homotopy
            \[ \begin{array}[c]{ccccc}
                      z_{i+5-\ell}(R_{\le s-1},\ast) & \xrightarrow{} & z_{i+5-\ell}(R_{\le s},\ast) & \to &  z_{i+5-\ell}(R_s,\ast) \to   \\
                      &&&&\\
                      \uparrow{\scriptstyle \cdot{M\vert_{R_{\le s-1}}}}&&\uparrow{\scriptstyle \cdot M\vert_{R_{\le s}}} &&\uparrow{\scriptstyle \cdot M\vert_{R_s}}  \\
                      &&&&\\
                       z_{i+5}^{{M}}(R_{\le s-1},\ast) & \xrightarrow{} & \ \ z_{i+5}^M(R_{\le s},\ast)  & \to&   z_{i+5}^M(R_s,\ast) \to  \ .\\   
                     \end{array}\]          
        In particular, these diagrams exist for $M$ being (a representative of) the classes $h^r, b\in A^\ast(R)$ that make up the definition of the map $\Phi^j_i$.               
  The result of the above remarks is a commutative diagram with long exact rows  
     \[ \begin{array}[c]{ccccc}
                     \to  A_{i}(R_{s},j+1)& \xrightarrow{}&\ A_{i}(R_{\le s-1},j )& \xrightarrow{}& A_{i}(R_{\le s},j)\ \to   \\
    &&&&\\
   \ \ \ \ \ \ \ \  \uparrow {\scriptstyle \Phi_i^{j+1}\vert_{R_s}}&&\ \ \ \ \ \ \  \  \uparrow{\scriptstyle \Phi_i^j\vert_{R_{\le s-1}}} &&\ \ \ \uparrow
    {\scriptstyle \Phi_i^j\vert_{R_{\le s}}}  \\
    &&&&\\
        \to  \bigoplus  A_{i-k}^{}(U_s,j+1) &\to &   \bigoplus A_{i-k}^{}(U_{\le s-1},j)     &\xrightarrow{} & \  \bigoplus A_{i-k}(U_{\le s},j)\ \to\ .\\
       \end{array}\]    
   Applying noetherian induction and the five--lemma, one is reduced to proving the claim for $R_s\to U_s$. But $R_s$ is isomorphic to the product $U_s\times R_u$ and the fibre $R_u$ is a linear variety (i.e., $R_u$ can be written as a finite disjoint union of affine spaces $\A^r$). Cutting up the fibre $R_u$ and using another commutative diagram with long exact rows, one is reduced to proving that $A_i(U_s,j)\cong A_{i+r}(U_s\times\A^r,j)$, which is the homotopy property for higher Chow groups \cite[Theorem 2.1]{B2}. This proves the claim, and hence (\rom1). 
        
\vskip0.3cm
\noindent
(\rom2) 
%
%
%
The point is that there is also a version in homology of (\rom1). That is, for any $j\in\NN$ there are isomorphisms
  \begin{equation}\label{isohomr}
      \Phi^h_j\colon\ \ \ \bigoplus H_{j-2k}(U)\ \xrightarrow{\cong}\ H_j(R)\ ,\end{equation}
      where $\Phi^h_j$ is defined as
      \[ \Phi^h_j:= \Bigl(    h^5\circ (p_U)^\ast, h^4\circ (p_U)^\ast, h^3\circ (p_U)^\ast, (b\cdot h)\circ (p_U)^\ast,  h^2\circ (p_U)^\ast, 
                                      b\circ (p_U)^\ast , h\circ (p_U)^\ast, (p_U)^\ast          \Bigr)\ .     
         \]
This is proven just as (\rom1), using homology instead of higher Chow groups. Cycle class maps fit into a commutative diagram
  \[ \begin{array}[c]{ccc}
      \bigoplus A_{i-k}(U)& \xrightarrow{\Phi_i}& A_i(R)\\
      &&\\
      \downarrow&&\downarrow\\
      &&\\
      \bigoplus H_{2i-2k}(U)& \xrightarrow{\Phi^h_{2i}}&\ \  H_{2i}(R)\ .\\
      \end{array}\]
As the horizontal arrows are isomorphisms, and the left vertical arrow is injective (lemma \ref{Utriv}),
the right vertical arrow is injective as well.
This proves statement (\rom2).
   \end{proof} 
  
  For later use, we record the following result:
  
  \begin{corollary}\label{noodd} One has
   \[ H_j(R)=0\ \ \ \hbox{for\ $j$\ odd}\ .\]
    \end{corollary}
  
  \begin{proof} The threefold $X$ has $H_{j-1}(X)=\QQ$ for any $j-1$ even, and the Grassmannian $Gr_1$ has $H_j(Gr_1)=0$ for $j$ odd. The exact sequence
    \[   H_j(Gr_1)\ \to\ H_j(U)\ \to H_{j-1}(X)\ \to\ H_{j-1}(Gr_1)\ \to\ \]
    implies that the open $U:=Gr_1\setminus X$ has no odd--degree cohomology. In view of the isomorphism (\ref{isohomr}), $R$ has no odd--degree homology either.
   \end{proof}

 Let us now turn to the fibration $F\to X$, where $F$ (but not $X$) is singular.

   \begin{proposition}\label{p1} Let $p_X\colon F\to X$ be as in proposition \ref{bcp}. Let $h^k\colon A_i(F)\to A_{i-k}(F)$ denote the operation of intersecting with a hyperplane section. 
  
  \noindent
  (\rom1) There are isomorphisms
   \[ \begin{split} &\Phi_0\colon   \ \ \    A_{0}(X)\ \xrightarrow{\cong}\ A_0(F)\ ,\\ 
                       &\Phi_1 \colon\ \ \  A_1(X)\oplus A_{0}(X)\ \xrightarrow{\cong}\ A_1(F)\ ,\\ 
                        & \Phi_2\colon \ \ \    A_2(X)\oplus A_1(X)  \oplus   A_0(X)^{\oplus 2}    \ \xrightarrow{\cong}\ A_2(F)\  \ ,\\
                          &\Phi_3\colon \ \ \  A_3(X)^{}
                             \oplus A_2(X)^{}\oplus A_1(X)^{\oplus 2}\oplus A_0(X)^{\oplus 2}  \ \xrightarrow{\cong}\ A_3(F)\  ,\\
                          &\Phi_4\colon \ \ \ A_3(X)\oplus A_2(X)^{\oplus 2}\oplus A_1(X)^{\oplus 2}\oplus A_0(X)^{\oplus 2}\ \xrightarrow{\cong}\ A_4(F)\ .\\
                                                      \end{split}\]
              The maps $\Phi_j$ are defined as
            \[  \begin{split}   &\Phi_0:= h^5\circ (p_X)^\ast        \ ,\\    
                            &\Phi_1 := \Bigl(  h^5\circ (p_X)^\ast, h^4\circ (p_X)^\ast\Bigr)\ ,\\
                                         & \Phi_2:=\Bigl(   h^5\circ (p_X)^\ast, h^4\circ (p_X)^\ast, h^3\circ (p_X)^\ast, (b\cdot h)\circ (p_X)^\ast    \Bigr)\ ,\\
                                      & \Phi_3:=\Bigl(    h^5\circ (p_X)^\ast, h^4\circ (p_X)^\ast, h^3\circ (p_X)^\ast, (b\cdot h)\circ (p_X)^\ast,  h^2\circ (p_X)^\ast, 
                                      b\circ (p_X)^\ast           \Bigr)\ ,\\
                                       & \Phi_4:=\Bigl(  h^4\circ (p_X)^\ast, h^3\circ (p_X)^\ast, (b\cdot h)\circ (p_X)^\ast,  h^2\circ (p_X)^\ast, 
                                      b\circ (p_X)^\ast , h\circ (p_X)^\ast  ,      (p_X)^\ast(-)\cap e      \Bigr)\ ,\\    
                                     \end{split} \]
                                     where $b\in A^2(F)$ is a class made explicit in the proof (and $b\colon A_i(R)\to A_{i-2}(R)$ denotes the operation of intersecting with $b$, and similarly for $b\cdot h$), and $e\in A_7(F)$ is a class made explicit in the proof (and the last $(p_X)^\ast(-)$ means pullback of operational Chow cohomology).

   \noindent
   (\rom2) The maps $\Phi_i$ induce isomorphisms of homologically trivial cycles
   \[   \Phi_i\colon\ \ \ \bigoplus A_{i-k}^{hom}(X)\ \xrightarrow{\cong}\ A_i^{hom}(F)\ .\]

    \end{proposition}
 
 \begin{proof} 
 
 \noindent
 (\rom1) The element $b$ is defined just as in proposition \ref{p2}:
   \[ b:=     (\iota_F)^\ast\Bigl( \bigl( Gr_1\times  c_2(T)\bigr)\vert_Q\Bigr)\ \ \ \in\ A^2(F)\ ,\]
   where $\iota_F\colon  F\hookrightarrow Q$ denotes the inclusion morphism, and $A^\ast(F)$ is operational Chow cohomology of the singular variety $F$.
   
   To define the element $e\in A_7(F)$, we return to the description of the fibres of $p_X\colon F\to X$ given in \cite[Section 7]{BCP}. By definition of the variety $F$, we have
    \[ F=\bigl\{ (x,y) \in Gr_1\times Gr_2^\vee\ \vert\ x\in X\ ,\ (x,y)\in \sigma\bigr\}\ ,\]
    where $\sigma\subset \PP\times\PP^\vee$ is the incidence divisor. As in \cite[section 7]{BCP}, given $x\in\PP$ (or $y\in\PP^\vee$) let  
   $ x_i:=\phi_i^{-1}(x)\in\PP(\wedge^2 V)$ (resp. $y_i:= \phi_i^\vee(y)\in\PP(\wedge^2 V^\vee)$). 
   For a point $\omega\in\PP(\wedge^2 V)$ (or in $\PP(\wedge^2 V^\vee)$), let $rk(\omega)$ denote the rank of $\omega$ considered as a skew--form; the rank of $\omega$ is either $2$ or $4$.
   As
      explained in loc. cit., the expression for $F$ can be rewritten as
    \[ \begin{split}  F&=\bigl\{  (x,y)\in \PP\times\PP^\vee\ \vert\ rk(x_i)=2 ,\ rk(y_2)=2 ,\ y\in\ H_{x_2}\bigr\}\ \\
                  &= \bigl\{ (x,y) \in \ \PP\times\PP^\vee\ \vert\     rk(x_i)=2 ,\ rk(y_2)=2 ,\  A_y\cap\ker(x_2)\not=\emptyset\bigr\}\   \subset\ \PP(\wedge^2 V)\times \PP(\wedge^2 V^\vee) .  \\
                  \end{split} \]
(Here, $A_y\subset V^\vee$ denotes the $2$--dimensional subspace corresponding to $y$, and $\ker(x_2)\subset V^\vee$ denotes the kernel of the skew--form $x_2\in\PP(\wedge^2 V)$.)
   The stratification of the fibres $F_x$ as given in loc. cit. can be done relatively over $X$. That is, we define
   \[ \ZZZ:= \bigl\{  (x,y)\ \in \PP\times\PP^\vee\ \vert\ (x,y)\in F\ ,\ A_y\subset\ker(x_2)\bigr\}\ \ \ \subset\ F\ .\]
   The intersection of $\ZZZ$ with a fibre $F_x$ is the variety $Z\cong\PP^2$ of \cite[Lemma 7.2]{BCP}, and so $\ZZZ\to X$ is a $\PP^2$--fibration.
   
  The complement $F^0:=F\setminus\ZZZ$ can be described as   
  \[ F^0=  \bigl\{  (x,y) \in \ \PP\times\PP^\vee\ \vert\     rk(x_i)=2 ,\ rk(y_2)=2 ,\  \dim (A_y\cap\ker(x_2))=1\bigr\}   \ .\]
   The natural morphism $ F^0\to X$ factors as 
   \[ F^0\ \to\ \WWW\ \to\ X\ ,\]
   where
   \[ \WWW:= \bigl\{ (x,s) \in X\times\PP(V^\vee)\ \vert\ s\in\PP(\ker(x_2))\bigr\}\ ,\]
   and $\WWW\to X$ is a $\PP^2$--fibration. As explained in loc. cit., over each $x\in X$ the morphism from $(F^0)_x$ to $\WWW_x$ is a fibration with fibres isomorphic to $\PP^3\setminus\PP^1$.
   Let $\WWW^\prime\subset\WWW$ be the divisor
   \[ \WWW^\prime:= \bigl\{ (x,s) \in X\times\PP(V^\vee)\ \vert\ s\in\PP(\ker(x_2))\cap h\bigr\}\ ,\]   
   where $h\subset\PP(V^\vee)$ is a hyperplane section. The morphism $\WWW^\prime\to X$ is a $\PP^1$--fibration. The class $e\in A_7(F)$ is now defined as
   \[ e:= (p_X\vert_{F^0})^{-1}(\WWW^\prime)\ \ \ \in\ A_7(F^0)\cong A_7(F)\ ,\]
  where $A_7(F^0)\cong A_7(F)$ for dimension reasons.
    
  We observe that $e\in A_7(F)$ is not proportional to the class of a hyperplane section $h\in A_7(F)$. (Indeed, let $x\in X$ and $w\in(\WWW\setminus\WWW^\prime)_x$ and let $\nu\in A^1((F^0)_x)$ be the tautological class with respect to the projective bundle structure of $(F^0)_x\to\WWW_x$. Then $C:=((F^0)_x)_w\cdot \nu^2\cong\PP^1$ is an effective curve disjoint from $e$, whereas $h\cap C$ has strictly positive degree.) 
  
  Since we know that the fibres $F_x$ have $A_4(F_x)\cong\QQ^2$ (proposition \ref{bcp}), it follows that
   \[ h\vert_{F_x}\ ,\ e\vert_{F_x}\ \ \ \in\ A_4(F_x) \]
   generate $A_4(F_x)$. (Here, $e\vert_{F_x}\in A_4(F_x)$ is defined as $\tau^\ast(e)\in A_4(F_x)$ where $\tau^\ast$ is the refined Gysin homomorphism \cite{F} associated to the regular morphism $\tau\colon x\hookrightarrow X$.)   
   
     The ($5$--dimensional) fibres $F_x$ of the fibration $p_X\colon F\to X$ thus verify
   \[ A_i(F_x)=  \begin{cases}     \QQ\cdot h^{5-i}\vert_{F_x}  &\ \hbox{if}\ i=0,1,5\ ,\\
                                       \QQ\cdot h^3\vert_{F_x} \oplus \QQ\cdot (b\cdot h)\vert_{F_x} &\ \hbox{if}\ i=2\ ,\\
                             \QQ\cdot h^2\vert_{F_x} \oplus \QQ\cdot b\vert_{F_x} &\ \hbox{if}\ i=3\ ,\\
                             \QQ\cdot h\vert_{F_x}\oplus \QQ\cdot e\vert_{F_x} &\ \hbox{if}\ i=4\ .\\
                                 \end{cases}\]

 We would like to prove proposition \ref{p1} following the strategy of proposition \ref{p2}, i.e. invoking higher Chow groups. The only delicate point is that $F$ is singular, and we need to make sense of the operation of ``capping with $h^k$ (or $b$)'' on higher Chow groups. Since this seems difficult\footnote{It is not clear whether operational Chow cohomology operates on higher Chow groups of a singular variety, which is a nuisance.}, we will prove proposition \ref{p1} without using higher Chow groups.
 
The piecewise triviality of the fibration $p_X$ means that there exist opens
   \[ F_0=F\setminus F_{\ge 1}\ ,\ \ \ X_0=X\setminus X_{\ge 1}\ \]
   such that $F_0$ is isomorphic to the product $X_0\times F_x$, and $F_{\ge 1}\to X_{\ge 1}$ is a piecewise trivial fibration (with fibre $F_x$).
   There is a commutative diagram with long exact rows  
     \[ \begin{array}[c]{ccccc}
                     \to  A_{i}(F_{\ge 1})& \xrightarrow{}&\ A_{i}(F )& \xrightarrow{}& A_{i}(F_{0})\ \to 0   \\
    &&&&\\
   \ \ \ \ \ \ \ \  \uparrow {\scriptstyle \Phi_i^{}\vert_{F_{\ge 1}}}&&\ \ \ \ \ \ \  \  \uparrow{\scriptstyle \Phi_i{{}}} &&\ \ \ \uparrow{\scriptstyle \Phi_i\vert_{F_0}}  \\
    &&&&\\
        \to    \bigoplus  A_{i-k}^{}(X_{\ge 1}) &\to &    \bigoplus A_{i-k}^{}(X_{})     &\xrightarrow{} & \  \bigoplus A_{i-k}(X_{0})\ \to\ 0\ .\\
       \end{array}\]    

The arrow $\Phi_i\vert_{F_0}$ is an isomorphism, because the fibre $F_x$ is a {\em linear variety\/} in the sense of \cite{Tot}, which implies (by \cite[Proposition 1]{Tot}, cf. also \cite[Theorem 4.1]{Tot2})  that the natural map
  \[  \bigoplus_{k+\ell=i}   A_k(M)\otimes A_\ell(F_x)\ \to\ A_i(M\times F_x) \]
  is an isomorphism for any variety $M$, and thus in particular for $M=X_0$. By noetherian induction, we may assume that $\Phi_i\vert_{F_{\ge 1}}$ is surjective. Contemplating the diagram, we find that the middle arrow $\Phi_i$ is also surjective.
  
 It remains to prove injectivity. 
 To this end, let us define a map
   \[  \begin{split} \Psi^\prime_i\colon\ \ \ A_i(F)\ &\to\ \bigoplus A_{i-k}(X)\ ,\\
                       a\ &\mapsto\ \Bigl(  (p_X)_\ast (a),  (p_X)_\ast (h\cap a),   (p_X)_\ast (h^2\cap a), (p_X)_\ast (h^2\cap a),\ldots,  (p_X)_\ast (h^5\cap a)\Bigr)\ .
                       \\
                       \end{split}\]
            Using the above--mentioned isomorphism
   \[   \bigoplus_{k+\ell=i}   A_k(X_0)\otimes A_\ell(F_x)\ \xrightarrow{\cong}\ A_i(X_0\times F_x) \ ,\]
one finds that  $ \Psi^\prime_i\vert_{F_0}\circ \Phi_i\vert_{F_0}$ is 
given by an invertible diagonal matrix. Dividing by some appropriate numbers, one can find $\Psi_i$ such that $\Psi_i\vert_{F_0}\circ \Phi_i\vert_{F_0}$ is the identity.

Using the projection formula, we see that there is a commutative diagram
     \[ \begin{array}[c]{ccccccc}
      \to&    \bigoplus  A_{i-k}^{}(X_{\ge 1}) &\to &    \bigoplus A_{i-k}^{}(X_{})     &\xrightarrow{} & \  \bigoplus A_{i-k}(X_{0})& \to \ 0\ \\
       &&&&&&\\
  & \ \ \ \ \ \ \ \  \uparrow {\scriptstyle \Psi_i^{}\vert_{F_{\ge 1}}}&&\ \ \ \ \ \ \  \  \uparrow{\scriptstyle \Psi_i{{}}} &&\ \ \ \uparrow{\scriptstyle \Psi_i\vert_{F_0}} & \\
    &&&&&&\\     
         \to & A_{i}(F_{\ge 1})& \xrightarrow{}&\ A_{i}(F )& \xrightarrow{}& A_{i}(F_{0}) & \to 0   \\
    &&&&&&\\
 &  \ \ \ \ \ \ \ \  \uparrow {\scriptstyle \Phi_i^{}\vert_{F_{\ge 1}}}&&\ \ \ \ \ \ \  \  \uparrow{\scriptstyle \Phi_i{{}}} &&\ \ \ \uparrow{\scriptstyle \Phi_i\vert_{F_0}} & \\
    &&&&&&\\
        \to &   \bigoplus  A_{i-k}^{}(X_{\ge 1}) &\to &    \bigoplus A_{i-k}^{}(X_{})     &\xrightarrow{} & \  \bigoplus A_{i-k}(X_{0})& \to \ 0\ .\\
         \end{array}\]    

We now make the following claim:

\begin{claim}\label{pol} For any given $i$, there exists a polynomial $p_i(x)\in\QQ[x]$ such that 
  \[ a=p_i(\Psi_i\circ \Phi_i)(a)  \ \ \ \forall  a\in  \bigoplus A_{i-k}^{}(X_{}) \ .\]
  \end{claim}
   
Clearly, the claim implies injectivity of $\Phi_i$. To prove the claim, we apply noetherian induction. Given  $a\in  \bigoplus A_{i-k}^{}(X_{}) $, we know that
$  \Psi_i\vert_{F_0}\circ \Phi_i\vert_{F_0}$ acts as the identity on the restriction $a\vert_{X_0}\in  \bigoplus A_{i-k}^{}(X_{0})$. It follows that we can write
  \begin{equation}\label{plug} a - (\Psi_i\circ \Phi_i)(a)=  b\ \ \ \hbox{in}\  \bigoplus A_{i-k}^{}(X_{})\ ,\end{equation}
  where $b$ is in the image of the pushforward map $ \bigoplus A_{i-k}^{}(X_{\ge 1}) \to \bigoplus A_{i-k}^{}(X_{}) $.
  By noetherian induction, we may assume the claim is true for the piecewise trivial fibration $F_{\ge 1}\to X_{\ge 1}$, and so there is a polynomial $q_i$ such that
  \[ b=q_i (\Psi_i\circ \Phi_i)(b)\ \ \ \hbox{in}\ \bigoplus A_{i-k}(X)\ .\]
  Plugging this in (\ref{plug}), we find that
  \[  a - (\Psi_i\circ \Phi_i)(a)=  q_i  (\Psi_i\circ \Phi_i) \Bigl(a -  (\Psi_i\circ \Phi_i)(a)\Bigr)      \ \ \ \hbox{in}\  \bigoplus A_{i-k}^{}(X_{})  \ .\]
  It follows that
    \[ a=p_i(\Psi_i\circ \Phi_i)(a)  \ \ \ \hbox{in}\  \bigoplus A_{i-k}^{}(X_{})  \ ,\]
   where the polynomial $p_i$ is defined as
   \[ p_i(x):= q_i(x) - x q_i(x) +x\ \ \in \QQ[x]\ .\]

 \vskip0.3cm
 \noindent
 (\rom2) As in proposition \ref{p2}, one can also prove a homology version of (\rom1). That is, for any $j\in\NN$ there are isomorphisms
  \begin{equation}\label{isohomf}
      \Phi^h_j\colon\ \ \ \bigoplus H_{j-2k}(X)\ \xrightarrow{\cong}\ H_j(F)\ ,\end{equation}
      where $\Phi^h_j$ is now defined as
      \[ \begin{split}   \Phi^h_j:= \Bigl(    h^5\circ (p_X)^\ast, h^4\circ (p_X)^\ast, h^3\circ (p_X)^\ast, (b\cdot h)\circ (p_X)^\ast, 
              h^2\circ (p_X)^\ast, 
                                      b\circ (p_X)^\ast,&\\
                                       h\circ (p_X)^\ast,   (p_X)^\ast(-)\cap e,   (p_X)^\ast&          \Bigr)\ . \\
                                       \end{split} \]
This is proven just as (\rom1), using homology instead of higher Chow groups. Cycle class maps fit into a commutative diagram
  \[ \begin{array}[c]{ccc}
      \bigoplus A_{i-k}(X)& \xrightarrow{\Phi_i}& A_i(F)\\
      &&\\
      \downarrow&&\downarrow\\
      &&\\
      \bigoplus H_{2i-2k}(X)& \xrightarrow{\Phi^h_{2i}}&\ \  H_{2i}(F)\ .\\
      \end{array}\]
Horizontal arrows being isomorphisms, this proves (\rom2).
 
 \end{proof}

  \begin{remark} Comparing propositions \ref{p2} and \ref{p1}, we observe that the only difference is the class $e\in A_7(F)$ appearing in proposition \ref{p1} but not in \ref{p2}. This ``extra class'' $e$ appears because of the singularities: the fibres of $p\colon Q\to G$ are smooth over the open $U\subset Gr_1$, but degenerate to singular fibres over $X\subset Gr_1$ (cf. remark \ref{sing}), and this causes an extra Weil divisor class $e$ to appear in the singular fibres. This observation will be key to the proof of theorem \ref{main}.
  \end{remark}

\section{Main result}     

\begin{theorem}\label{main} Let $X,Y$ be a pair of GPK$^3$ double mirrors. Then there is an isomorphism 
   \[ h(X)\cong h(Y)\ \ \ \hbox{in}\ \MM_{{\rm rat}}\ .\]
  \end{theorem} 
     
  \begin{proof} The proof is a four--step argument, which exploits that the threefolds 
    \[   \begin{split}  X&:=  Gr_1\cap Gr_2\ \ \subset\ \PP\ ,\\
                                Y&:=Gr_1^\vee\cap Gr_2^\vee\ \ \subset\ \PP^\vee\ \\
                                \end{split}\]
     are geometrically related as in proposition \ref{bcp}. In essence, the argument is similar to the proof that $X,Y$ are L--equivalent \cite[Section 7]{BCP}, by applying ``cut and paste'' to the diagram of proposition \ref{bcp}. 
Here is an overview of the proof. Let 
   \[ Q:=  \sigma\times_{\PP\times\PP^\vee} (Gr_1\times Gr_2^\vee)  \]
  be the $11$--dimensional intersection as in proposition \ref{bcp}. Assuming $Q$ is non--singular, we prove there exist isomorphisms of Chow groups
    \begin{equation}\label{overview} A^i_{hom}(X)\ \xrightarrow{\cong}\ A^{i+4}_{hom}(Q)\  \ \ \hbox{for\ all\ }i\ .\end{equation}
  This is done in step 1 (for $i=2,3$) and step 2 (for $i=0,1$), and relies on the isomorphisms for the piecewise trivial fibrations established in the prior section. In step 3, the isomorphism (\ref{overview}) is upgraded to an isomorphism of Chow motives 
    \[ h^3(X) \cong h^{11}(Q)(-4)\ \ \ \hbox{in}\ \MM_{\rm rat}\ .\]
    As $X$ and $Y$ are symmetric, this implies an isomorphism of Chow motives $h^3(X)\cong h^3(Y)$, and hence also $h(X)\cong h(Y)$. Finally, in step 4 we show that we may ``spread out'' this isomorphism to all 
    GPK$^3$ double mirrors.

  \vskip0.5cm
 \noindent 
  {\it Step 1: an isomorphism of Chow groups.\/}  
  In this first step, we assume the $Gr_i$ are sufficiently general, so that $Q$ is non--singular (there is no loss in generality; the degenerate case where $Q$ may be singular will be taken care of in step 4 below). The goal of this first step will be to construct an isomorphism between certain Chow groups of $X$ and $Y$:
  
  \begin{proposition}\label{goal1} There exist correspondences $\Gamma\in A^{7}(X\times Q)$, $\Psi\in A^{7}(Y\times Q)$ inducing isomorphisms
    \[ \begin{split} \Gamma_\ast\colon\ \ \ A^i_{hom}(X)\ \xrightarrow{\cong}\ A^{i+4}_{hom}(Q)\ \ \ \hbox{for}\ i=2,3 ,\\
                            \Psi_\ast\colon\ \ \   A^i_{hom}(Y)\ \xrightarrow{\cong}\ A^{i+4}_{hom}(Q)\ \ \ \hbox{for}\ i=2,3 .\\
           \end{split}\]   
    \end{proposition}

   Before proving this proposition, let us first establish two lemmas (in these lemmas, we continue to assume $X,Y$ are sufficiently general, so that $Q$ is smooth):
   
   \begin{lemma}\label{l1} The pushforward map
     \[ (\iota_F)_\ast\colon\ \ \ A_i^{hom}(F)\ \to\ A_i^{hom}(Q) \]
     is surjective, for all $i$.
     \end{lemma}
     
     \begin{proof} As before, let $R$ denote the open complement $R:=Q\setminus F$. There is a commutative diagram with exact rows
        \[ \begin{array}[c]{cccccccc}
     &\to & A_i(F) &\to& A_i(Q) &\to& A_i(R)&\to 0\\
     &&&&&&&\\
     &&\downarrow&&\downarrow&&\downarrow&\\
  &&&&&&&\\
   0&\to &  H_{2i}(F,\QQ) &\to& H_{2i}(Q,\QQ) &\to&\ \  H_{2i}(R,\QQ)\ ,&\\  
   \end{array}\]
   where vertical arrows are cycle class maps. Here, the lower left entry is $0$ because $R$ has no odd--degree homology (corollary \ref{noodd}). The lemma follows from the fact that the right vertical arrow is injective, which is proposition \ref{p2}(\rom2).    
      \end{proof}

    \begin{lemma}\label{l2} Let $e\in A_7(F)$ be as in proposition \ref{p1}. The composition
    \[   A^i_{hom}(X)\ \xrightarrow{(p_X)^\ast}\ A^{i}_{hom}(F)\ \xrightarrow{ (-)\cap e}\ A_{7-i}^{hom}(F)\ \xrightarrow{(\iota_F)_\ast}\ A_{7-i}^{hom}(Q)=A^{i+4}_{hom}(Q) \]
    is an isomorphism for $i=2,3$. (As before, $A^\ast(F)$ denotes operational Chow cohomology of the singular variety $F$.)
        \end{lemma}  
        
      \begin{proof} Let us treat the case $i=3$ in detail. Proposition \ref{p1} gives us an isomorphism
       \[     \Phi \colon \ \ \ A_3(X)\oplus A_2(X)^{\oplus 2}\oplus A_1(X)^{\oplus 2}\oplus A_0(X)^{\oplus 2}\ \xrightarrow{\cong}\ A_4(F)\ ,\\       \]
        where $\Phi:=\Phi_4$ is defined as
     \[ \Phi= \Bigl(  h^4\circ (p_X)^\ast, h^3\circ (p_X)^\ast, (b\cdot h)\circ (p_X)^\ast,  h^2\circ (p_X)^\ast, 
                                      b\circ (p_X)^\ast , h\circ (p_X)^\ast  ,      (p_X)^\ast(-)\cap e      \Bigr)\ .\\ \]   
                                      
              We want to single out the part in $A_4(F)$ coming from the ``extra class'' $e\in A_7(F)$.
        That is, we write the isomorphism $\Phi$ as a decomposition
       \begin{equation}\label{aperp}   A_4^{}(F) =  A^\perp \oplus A  \ ,\end{equation}
       where
        \[  \begin{split}
        A&:=  (p_X)^\ast A_0{}(X)\cap e\ ,\\     
            A^\perp &:= \ima \Bigl (   A_3(X)\oplus A_2(X)^{\oplus 2}\oplus A_1(X)^{\oplus 2}\oplus A_0(X)       \ \xrightarrow{\Phi^\perp}\ A_4(F)\Bigr)\ ,\\
        \Phi^\perp&:= \Bigl(    h^4\circ (p_X)^\ast, h^3\circ (p_X)^\ast, (b\cdot h)\circ (p_X)^\ast,  h^2\circ (p_X)^\ast, 
                                      b\circ (p_X)^\ast , h\circ (p_X)^\ast            \Bigr)\ .\\ \end{split}\]
        The decomposition (\ref{aperp}) also exists in cohomology, and so there is an induced decomposition
          \begin{equation}\label{aperph}   A_4^{hom}(F) =  A^\perp_{hom} \oplus A_{hom}  \ ,\end{equation}        
   where we put
   \[  A_{hom}:= A\cap A_4^{hom}(F)\ ,\ \ \ A_{hom}^\perp:= A^\perp\cap A_4^{hom}(F)\ .\]

   We now claim that there is a commutative diagram with exact rows
   \begin{equation}\label{claimcom}   \begin{array}[c]{ccccc}
                       A_{4}(R,1)& \xrightarrow{\delta}&\ A_{4}(F)& \xrightarrow{(\iota_F)_\ast}& A_{4}(Q)\ \to   \\
    &&&&\\
    \uparrow {\scriptstyle \Phi_4^1}&&\ \ \ \ \ \ \ \ \ \  \uparrow{\scriptstyle \Phi^\perp} &&\uparrow\\
    &&&&\\
          \bigoplus A_{k}^{}(U,1)  &\xrightarrow{\delta_U} &   \bigoplus A_{k}^{}(X)      &\xrightarrow{\iota_\ast} &  A_{0}(G)\oplus A_1(G)^{\oplus 2}\ \to\ ,\\
       \end{array}\end{equation}     
       where $\Phi_4^1$ is the isomorphism of proposition \ref{p2} (and we use the shorthand $G:=Gr_1$).   
   
   Granting this claim, let us prove the lemma for $i=3$. The kernel of $(\iota_F)_\ast$ equals the image of the arrow $\delta$. Since $\Phi_4^1$ is an isomorphism, the image  $\ima \delta$ is contained in $\ima \Phi^\perp=:A^\perp$. In view of the decomposition (\ref{aperp}), this implies injectivity
   \begin{equation}\label{inj}  (\iota_F)_\ast\colon \ \ A\ \hookrightarrow\ A_4(Q)\ ,\end{equation}
   i.e. the composition of lemma \ref{l2} is injective for $i=3$.
   
   To prove surjectivity, let us consider a class
    $ b \in A^\perp_{hom}$.
  We know (from proposition \ref{p1}(\rom2)) that 
    \[ b= \Phi^\perp (\beta)\ , \ \ \ \beta\in\   A_1^{hom}(X)^{\oplus 2}\oplus A_0^{hom}(X)\ .\]
    Referring to diagram (\ref{claimcom}), we see that $\iota_\ast(\beta)$ must be $0$ (for the Grassmannian $G$ has trivial Chow groups). It follows that $\beta$ is in the image of $\delta_U$, and so $b\in \ima \delta$. This shows that 
    \[ (\iota_F)_\ast ( A^\perp_{hom})=0\ ,\]
    and hence, in view of the decomposition (\ref{aperph}), that
    \[      (\iota_F)_\ast ( A_{hom})=   (\iota_F)_\ast ( A_4^{hom}(F))\ .     \]
    On the other hand, we know that $ (\iota_F)_\ast ( A_4^{hom}(F))= A_4^{hom}(Q)$  (lemma \ref{l1}), and so we get a surjection
    \begin{equation}\label{sur}     (\iota_F)_\ast\colon \ \ A_{hom}\ \twoheadrightarrow\ A_4^{hom}(Q)\ .\end{equation}
    Combining (\ref{inj}) and (\ref{sur}), we see that the composition of lemma \ref{l2} is an isomorphism for $i=3$.

     It remains to establish the claimed commutativity of diagram (\ref{claimcom}).  
      The morphism $p\colon Q\to G$ is equidimensional of relative dimension $5$, and so there is a commutative diagram of complexes (where rows are exact triangles)
      \[ \begin{array}[c]{cccccc}
            z_{i+5}(F,\ast) &\to& z_{i+5}(Q,\ast) &\to& z_{i+5}(R,\ast) &\to\\
            &&&&&\\
            \uparrow{\scriptstyle (p_X)^\ast}&& \uparrow{\scriptstyle p^\ast} &&       \uparrow{\scriptstyle (p_U)^\ast}  \\
            &&&&&\\
              z_{i}(X,\ast) &\to& z_{i}(G,\ast) &\to& z_{i+5}(U,\ast) &\to\\
              \end{array}\]
        Also, given a codimension $k$ subvariety $M\subset Q$, let $z_i^M(Q,\ast)\subset z_i(Q,\ast)$ denote the subcomplex formed by cycles in general position with respect to $M$. The inclusion $z_i^M(Q,\ast)\subset z_i(Q,\ast)$ is a quasi--isomorphism \cite[Lemma 4.2]{B2}.
           The diagram
           \[     \begin{array}[c]{cccccc}           
             z_{i+5-k}(F,\ast) &\to& z_{i+5-k}(Q,\ast) &\to& z_{i+5-k}(R,\ast) &\to\\           &&&&&\\
           &&\uparrow{\scriptstyle \cdot M}&&\uparrow{\scriptstyle \cdot M\vert_R}\\
           &&&&&\\
             && z_{i+5}^M(Q,\ast) &\to& z_{i+5}^M(R,\ast) &\to\\
                  &&&&&\\
                                  &&\downarrow{\scriptstyle \simeq}&&\downarrow{\scriptstyle \simeq}\\
                                &&&&&\\
            z_{i+5}(F,\ast) &\to& z_{i+5}(Q,\ast) &\to& z_{i+5}(R,\ast) &\to\\
            \end{array}\]
            (where $\simeq$ indicates quasi--isomorphisms) defines an arrow in the homotopy category
            \[  f_M\colon\ \ z_{i+5}(F,\ast)\ \to\ z_{i+5-k}(F,\ast)\ .\]
            (The arrow $f_M$ represents ``intersecting with $M$''.)
            On the other hand, let $g\colon\wt{F}\to F$ be a resolution of singularities, and let $\wt{M}:=(\iota_F\circ g)^\ast(M)\in A^k(\wt{F})$. The projection formula for 
            higher Chow groups \cite[Exercice 5.8(\rom1)]{B2} gives a commutative diagram up to homotopy
            \[ \begin{array}[c]{ccc}
                      z_{i+5-k}(\wt{F},\ast) & \xrightarrow{(\iota_F\circ g)_\ast} & z_{i+5-k}(Q,\ast)\\
                      &&\\
                      \uparrow{\scriptstyle \cdot\wt{M}}&&\uparrow{\scriptstyle \cdot M}\\
                      &&\\
                       z_{i+5}^{\vert\wt{M}\vert}(\wt{F},\ast) & \xrightarrow{(\iota_F\circ g)_\ast} & \ \ z_{i+5}^M(Q,\ast)\ ,\\   
                     \end{array}\]          
      and so there is also a commutative diagram up to homotopy
              \[ \begin{array}[c]{ccc}
                      z_{i+5-k}(\wt{F},\ast) & \xrightarrow{} & z_{i+5-k}(F,\ast)\\
                      &&\\
                      \uparrow{\scriptstyle \cdot\wt{M}}&&\uparrow{\scriptstyle f_M}\\
                      &&\\
                       z_{i+5}^{}(\wt{F},\ast) & \xrightarrow{} & \ \ z_{i+5}(F,\ast)\ .\\   
                     \end{array}\]                
             In particular, this shows that
             \[ f_M =  (\iota_F)^\ast(M)\cap (-)\colon\ \ \ A_{i+5}(F)\ \to\ A_{i+5-k}(F) \ ,\]
             where we consider $  (\iota_F)^\ast(M)\in A^k(F)$ as an element in operational Chow cohomology.      
             
             Combining the above remarks, one obtains a commutative diagram with long exact rows
             \[ \begin{array}[c]{ccccc}
                       A_{i+5-k}(R,1)& \xrightarrow{}&\ A_{i+5-k}(F)& \xrightarrow{(\iota_F)_\ast}& A_{i+5-k}(Q)\ \to   \\
    &&&&\\
   \ \ \ \uparrow {\scriptstyle \cdot M\vert_R}&&\ \ \ \ \ \ \ \ \ \  \ \ \uparrow{\scriptstyle (\iota_F)^\ast(M)\cap (-)} &&\uparrow\\
    &&&&\\
            A_{i+5}(R,1)& \xrightarrow{}&\ A_{i+5}(F)& \xrightarrow{(\iota_F)_\ast}& A_{i+5}(Q)\ \to   \\    
            &&&&\\
   \ \ \  \uparrow {\scriptstyle (p_U)^\ast}&&\ \  \uparrow{\scriptstyle (p_X)^\ast} &&\ \ \ \uparrow{\scriptstyle p^\ast}\\
    &&&&\\
     A_{i}^{}(U,1)  &\to &    A_{i}^{}(X)      &\xrightarrow{\iota_\ast} &  A_{i}(G)\ \to\ .\\
       \end{array}\]     
      In particular, these diagrams exist for $M$ being (a representative of) the classes $h^j, b\cdot h, b \in A^\ast(Q)$ that make up the definition of the map 
      $\Phi^\perp$.               
    It follows there is a commutative diagram with long exact rows
    \[ \begin{array}[c]{ccccc}
                       A_{4}(R,1)& \xrightarrow{}&\ A_{4}(F)& \xrightarrow{(\iota_F)_\ast}& A_{4}(Q)\ \to   \\
    &&&&\\
    \uparrow {}&&\ \ \ \ \ \ \ \ \ \  \uparrow{\scriptstyle \Phi^\perp} &&\uparrow\\
    &&&&\\
          \bigoplus A_{k}^{}(U,1)  &\to &   \bigoplus A_{k}^{}(X)     &\xrightarrow{} &  \bigoplus A_{k}(G)\ \to\ ,\\
       \end{array}\]

The $i=2$ case of the lemma is proven similarly: using proposition \ref{p2}, we can write
          \begin{equation}\label{adecomp}   A_5^{hom}(F) =   A_5^\perp\oplus  (p_X)^\ast A^2_{hom}(X)\cap e \ .\end{equation}
        Here $A_5^\perp$ is
  \[ A_5^\perp = \ima \Bigl (     A^3_{hom}(X)^{}\oplus   A^2_{hom}(X)^{}   \ \xrightarrow{\Phi^\perp}\ A_5(F)\Bigr)\ ,\]
      where $\Phi^\perp$ is defined as
      \[ \Phi^\perp:= \Bigl(   (p_X)^\ast(-),  h\cap (p_X)^\ast(-)\Bigr)\ .\]
                       As above, there is a commutative diagram with long exact rows
                           \[ \begin{array}[c]{ccccc}
                       A_{5}(R,1)& \xrightarrow{}&\ A_{5}(F)& \xrightarrow{(\iota_F)_\ast}& A_{5}(Q)\ \to   \\
    &&&&\\
    \uparrow {}&&\ \ \ \ \ \ \ \ \ \  \uparrow{\scriptstyle \Phi^\perp} &&\uparrow\\
    &&&&\\
           A_{0}^{}(U,1)\oplus A_1(U,2)^{}  &\to &    A_{0}^{}(X)\oplus A_1(X)^{}      &\xrightarrow{\iota_\ast} &  A_{0}(G)\oplus A_1(G)^{}\ \to\ .\\
       \end{array}\]     
       As above, chasing this diagram we conclude that
       \[ (\iota_F)_\ast     (A_5^\perp)=  (\iota_F)_\ast \Phi^\perp \Bigl(   A_{0}^{hom}(X)\oplus A_1^{hom}(X)^{} \Bigr)  =0\ \ \ \hbox{in}\  A_5(Q)\ .\]
       It follows that the restriction of $(\iota_F)_\ast$ to the second term of the decomposition (\ref{adecomp}) induces an isomorphism
        \[ (\iota_F)_\ast        \colon\ \          (p_X)^\ast A^2_{hom}(X)\cap e    \ \xrightarrow{\cong}\ A_5^{hom}(Q)\ . \]
                                 \end{proof}

Let us now proceed to prove proposition \ref{goal1}. We will construct the correspondence $\Gamma$ (the construction of $\Psi$ is only notationally different, the roles of $X$ and $Y$ being symmetric). The variety $X$ is smooth, and the variety $Q$ of proposition \ref{bcp} is also smooth, by our generality assumptions. The variety $F$, however, is definitely singular (remark \ref{sing}), and so we need to desingularize.
 Let $g\colon\wt{F}\to F$ be a resolution of singularities. 
 We let $\bar{e}\subset\wt{F}$ denote the strict transform of $e\subset F$, and $\wt{e}\to\bar{e}$ a resolution of singularities, and we write $\tau\colon \wt{e}\to\wt{F}$ for the composition of the resolution and the inclusion morphism.
 The correspondence $\Gamma$ will be defined as
   \[  \Gamma:= \Gamma_{\iota_{{F}}}\circ \Gamma_g \circ  (\Gamma_{{\tau}} \circ {}^t \Gamma_{{\tau}}) \circ {}^t \Gamma_g \circ {}^t \Gamma_{{p_X}}\ \ \ \in \
       A^7(X\times Q)\ . \]
   By definition, the action of $\Gamma$ decomposes as
      \[  \begin{split} \Gamma_\ast\colon\ \ A^i(X)\ \xrightarrow{(p_X)^\ast} \ A^i(F) \ \xrightarrow{g^\ast}\ A^i(\wt{F})
     \ \xrightarrow{\cdot \bar{e}}\ A^{i+1}(\wt{F})=A_{7-i}(\wt{F})\ \xrightarrow{g_\ast} &\\ A_{7-i}(F)\ \xrightarrow{(\iota_F)_\ast}\ A_{7-i}(Q)= A^{i+4}(Q)\ . &\\
      \end{split}
        \]
         (Here $A^\ast(F)$ refers to Fulton--MacPherson's operational Chow cohomology \cite{F}.)
         The projection formula for operational Chow cohomology (theorem \ref{oper}) ensures that for any $b\in A^i(F)$ one has
         \[ g_\ast ( g^\ast(b)\cdot \bar{e})= g_\ast ( g^\ast(b)\cap \bar{e})=b\cap e\ \ \ \hbox{in}\ A_{7-i}(F)\ ,\]
         and so the action of $\Gamma$ simplifies to
       \[    \Gamma_\ast\colon\ \ A^i(X)\ \xrightarrow{(p_X)^\ast} \ A^i(F) 
     \ \xrightarrow{(\ )\cap {e}}\  A_{7-i}(F)\ \xrightarrow{(\iota_F)_\ast}\ A_{7-i}(Q)= A^{i+4}(Q)\ . 
         \]   
  Lemma \ref{l2} ensures that $\Gamma_\ast$ induces isomorphisms
   \[ \Gamma_\ast\colon\ \ A^i_{hom}(X)\ \xrightarrow{\cong}\ A^{i+4}_{hom}(Q)\ \ \ (i=2,3)\ ,\]
  and so we have proven proposition \ref{goal1}.

\vskip0.5cm
\noindent
{\it Step 2: Trivial Chow groups.\/} In this step, we study the Chow groups of the incidence variety $Q$. We continue to assume that $Q$ is smooth, just as in step 1.
The goal of step 2 will be to show that many Chow groups of $Q$ are trivial:

\begin{proposition}\label{goal2} We have
  \[ A^i_{hom}(Q)=0\ \ \ \hbox{for\ all\ }i\not\in\{6,7\}\ .\]
  \end{proposition}
  (This means that $\hbox{Niveau}(A^\ast(Q))\le 3$ in the language of \cite{moi}, i.e. the $11$--dimensional variety $Q$ motivically looks like a variety of dimension $3$.)
  
  \begin{proof} Suppose we can prove that
    \begin{equation}\label{<}  A_j^{hom}(Q)=0\ \ \ \hbox{for}\ j< 4\ .\end{equation}
       Applying the Bloch--Srinivas argument \cite{BS}, \cite[Remark 1.8.1]{moi} to the smooth projective variety $Q$, this implies that
       \[ A^i_{AJ}(Q)=0\ \ \ \hbox{for\ all\ }i\not\in\{6,7\}\ .\]
But $Q$ has no odd--degree cohomology except for degree $11$ (proposition \ref{H11} below), and so there is equality $A^i_{AJ}(Q)=A^i_{hom}(Q)$ for all $i\not=6$. That is, to prove proposition \ref{goal2} one is reduced to proving (\ref{<}).

  There is an exact sequence
   \[  A_j(R,1)\ \xrightarrow{\delta}\ A_j(F)\ \xrightarrow{(\iota_F)_\ast}\ A_j(Q)\ \to\ \ \ ,\]
   and we have seen (lemma \ref{l1}) that $(\iota_F)_\ast A_j^{hom}(F)=A_j^{hom}(Q)$. Thus, to prove the vanishing (\ref{<}) it only remains to show that
   \begin{equation}\label{want}  A_j^{hom}(F)\ \ \subset\ \ima \, \delta\ \ \ \hbox{for\  }j< 4\ .\end{equation}
   
   The inclusion (\ref{want}) is proven by the same argument as that of lemma \ref{l2}. That is, we observe that propositions \ref{p1} and \ref{p2} give us isomorphisms
   \begin{equation}\label{isos} \begin{split}   \Phi_j\colon\ \ \bigoplus A_{j-k}^{}(X)        \ &\xrightarrow{\cong} \ A_j^{}(F)\ ,\\
                          \Phi^1_j\colon\ \ \bigoplus A_{j-k}^{}(U,1)        \ &\xrightarrow{\cong} \ A_j^{}(R,1)\ ,\\
                          \end{split} \end{equation}
               where $\Phi_j, \Phi_j^1$ are defined as       
       \[ \begin{split}   \Phi_j&= \begin{cases}  h^5 \circ (p_X)^\ast     &\ \ \hbox{if}\ j=0\ ,\\
                                                                      \sum_{k=0}^1  h^{5-k}\circ (p_X)^\ast   &\ \ \hbox{if}\ j=1\ ,\\    
      \sum_{k=0}^2  h^{5-k}\circ (p_X)^\ast + b \circ (p_X)^\ast     &\ \ \hbox{if}\ j=2\ ,\\     
          \sum_{k=0}^3  h^{5-k}\circ (p_X)^\ast        + b \circ (p_X)^\ast   +   (b\cdot h)\circ (p_X)^\ast     &\ \ \hbox{if}\ j=3\ ,\\   
         \end{cases} \\
                               \Phi^1_j&= \begin{cases}  h^5 \circ (p_U)^\ast     &\ \ \hbox{if}\ j=0\ ,\\
                                                                      \sum_{k=0}^1  h^{5-k}\circ (p_U)^\ast   &\ \ \hbox{if}\ j=1\ ,\\    
 \sum_{k=0}^2  h^{5-k}\circ (p_U)^\ast + b \circ (p_U)^\ast     &\ \ \hbox{if}\ j=2\ ,\\     
 \sum_{k=0}^3  h^{5-k}\circ (p_X)^\ast + b \circ (p_U)^\ast  + (b\cdot h)\circ (p_U)^\ast     &\ \ \hbox{if}\ j=3\ .\\   
         \end{cases} \\       
         \end{split}\]      
         
   In particular, we observe that for each $j\le 3$, the isomorphisms $\Phi_j, \Phi^1_j$ of (\ref{isos}) have the same number of direct summands on the left--hand side (i.e., the ``extra class'' $e\in A_7(F)$ does not appear).
   For each $j\le 3$, we can construct a commutative diagram with exact rows
   \[ \begin{array}[c]{ccccc}    
    A_j(R,1)& \xrightarrow{\delta}&\ A_j(F)& \xrightarrow{(\iota_F)_\ast}& A_j(Q)\ \to   \\
    &&&&\\
   \ \  \uparrow {\scriptstyle \Phi_j^1}&&\ \  \ \uparrow{\scriptstyle \Phi_j} &&\uparrow\\
    &&&&\\
    \bigoplus A_{j-k}^{}(U,1)  &\xrightarrow{\delta_U} & \bigoplus   A_{j-k}^{}(X)      &\xrightarrow{\iota_\ast} &  \bigoplus A_{j-k}(Gr_1)\\
       \end{array}\]   
     (the commutativity of this diagram is checked as in the proof of lemma \ref{l2}).    
     
     Let $a\in A_j^{hom}(F)$, for $j\le 3$. Then $a=\Phi_j(\alpha)$ for some $\alpha\in\bigoplus A_{j-k}^{hom}(X)$ (proposition \ref{p1}(\rom2)). But then $\iota_\ast(\alpha)=0$, since the Grassmannian $Gr_1$ has trivial Chow groups. Using the above diagram, it follows that $\alpha$ is in the image of $\delta_U$, and hence
       $a\in\ima \delta$. This proves (\ref{want}) and hence proposition \ref{goal2}.    
                                   
  (NB: in fact, the above argument does {\em not\/} need that $\Phi^1_j$ is an isomorphism; we merely need the fact that a map $\Phi^1_j$ fitting into the above commutative diagram exists.)
 
 To close step 2, it only remains to prove the following proposition:
 
 \begin{proposition}\label{H11} Assume $j$ is odd and $j\not=11$. Then
    \[  H_j(Q)=0\ .\]
  \end{proposition}
  
  \begin{proof} For $j$ odd and different from $11$, there is a commutative diagram with exact row
     \[ \begin{array}[c]{cccccc}    
    H_{j+1}(R)& \xrightarrow{\delta}&\ H_j(F)& \xrightarrow{(\iota_F)_\ast}& H_j(Q) & \to 0  \\
    &&&&&\\
    \uparrow {\scriptstyle \Phi_{j+1}^h}&& \uparrow{\scriptstyle \Phi^h_j} &&&\\
    &&&&&\\
    \bigoplus H_{j+1-2k}^{}(U)  &\to & \bigoplus   H_{j-2k}^{}(X) \ .     & & \ \   & \\
       \end{array}\]    
     Here $\Phi^h_j$ and $\Phi^h_{j+1}$ are the isomorphisms of (\ref{isohomf}) resp. (\ref{isohomr}).  
       The righthand $0$ is because $H_j(R)=0$ for $j$ odd (corollary \ref{noodd}). 
       
       We observe that $j\not=11$ implies that $k\not=4$ (the only odd homology of $X$ is $H_3(X)$), which means that the ``extra class'' $e\in A_7(F)$ does not intervene in the map $\Phi^h_j$. It follows that there are the same number of direct summands in the isomorphisms $\Phi^h_j, \Phi^h_{j+1}$ (they are both defined in terms of $h^k$ and $b$). 
    Observing that a Grassmannian does not have odd--degree cohomology, we thus see that the lower horizontal arrow is surjective.   
          The diagram now shows that $\delta$ is surjective, and hence $(\iota_F)_\ast $ is the zero--map. The proposition is proven. 
        \end{proof} 
 
 This ends the proof of proposition \ref{goal2}.
  \end{proof}

 \vskip0.5cm
 \noindent
 {\it Step 3: an isomorphism of motives.\/} We continue to assume (as in steps 1 and 2) that $X,Y$ are general, so that $Q$ is smooth.
 
 The assignment $\pi^3_X:=\Delta_X-\pi^0_X-\pi^2_X-\pi^4_X-\pi^6_X$ (where the K\"unneth components $\pi^j_X$, $j\not=3$ are defined using hyperplane sections) defines a motive $h^3(X)$ such that there is a splitting
   \[ h(X)=  \mathds{1} \oplus \mathds{1}(1)\oplus h^3(X) \oplus \mathds{1}(2) \oplus \mathds{1}(3) \ \ \ \hbox{in}\ \MM_{\rm rat}\ ,\]
   where $\mathds{1}$ is the motive of a point. It follows that
     \begin{equation}\label{h3} A^i(h^3(X))= A^i_{hom}(X)\ \ \ \forall i\ .\end{equation}
 
 The variety $Q$ is a smooth ample divisor in the product $P:=Gr_1\times Gr_2$. The product $P$ has trivial Chow groups, and hence in particular verifies the standard conjectures. It follows that $P$ admits a (unique) Chow--K\"unneth decomposition $\{\pi^i_P\}$, and that there exist correspondences $C^j\in A^{12+j}(P\times P)$ such that 
  \[ (C^j)_\ast\colon\ \  H^{12+j}(P)\ \to\ H^{12-j}(P) \]
  is inverse to
  \[ \cup Q^j\colon\ \ H^{12-j}(P)\ \to\ H^{12+j}(P)\ .\]
  (The correspondences $C^j\in A^{12+j}(P\times P)$ are well--defined, as rational and homological equivalence coincide on $P\times P$.)
  
 Let $\tau\colon Q\to P$ denote the inclusion morphism. One can construct a Chow--K\"unneth decomposition for $Q$, by setting
  \[    \pi^i_Q :=  \begin{cases}  {}^t \Gamma_\tau   \circ (\Gamma_\tau\circ{}^t \Gamma_\tau)^{\circ 11-i} \circ C^{12-i}  \circ  \pi_P^{i}  \circ \Gamma_\tau   &  \ \ \ \hbox{if\ } i<11\ ,\\
                                                 {}^t \pi^{22-i}_Q & \ \ \ \hbox{if\ } i>11\ ,\\  
                                                 \Delta_Q -\sum_{j\not= 11} \pi^j_Q \ \ \ \ \in\ A^{11}(Q\times Q)& \ \ \ \hbox{if\ }i=11\ .
                  \end{cases}\]
         (To check this is indeed a Chow--K\"unneth decomposition, one remarks that         
         \[   (\Gamma_\tau\circ{}^t \Gamma_\tau)^{\circ 12-i} \circ C^{12-i}  \circ  \pi_P^{i}   = \pi_P^i\ \ \ \hbox{in}\ H^{24}(P\times P)\ ,\]
         and because $P\times P$ has trivial Chow groups one has the same equality modulo rational equivalence:
            \[   (\Gamma_\tau\circ{}^t \Gamma_\tau)^{\circ 12-i} \circ C^{12-i}  \circ  \pi_P^{i}   = \pi_P^i\ \ \ \hbox{in}\ A^{12}(P\times P)\ .\]
       It is now readily checked that $\{\pi^i_Q\}$ verifies $\pi^i_Q\circ\pi^j_Q=\delta_{ij} \pi^i_Q$ in $A^{11}(Q\times Q)$, where $\delta_{ij}$ is the Kronecker symbol.)

  Setting $h^i(Q):=(Q,\pi^i_Q,0)$, this induces a decomposition of the motive of $Q$ as
   \[ h(Q) =\bigoplus \mathds{1}(\ast)\oplus h^{11}(Q)\ \ \ \hbox{in}\ \MM_{\rm rat}\ ,\]
  and hence one has
  \begin{equation}\label{h11}  A^i(h^{11}(Q))= A^i_{hom}(Q)\ \ \ \forall i\ .\end{equation}   
  
  We now consider the homomorphism of motives
  \[ \Gamma\colon\ \ \ h^3(X)\ \to\ h^{11}(Q)(-4)\ \ \  \hbox{in}\ \MM_{\rm rat}\ ,\]
where $\Gamma\in A^7(X\times Q)$ is as in step 1. We have seen in steps 1 and 2 that there are isomorphisms
 \[ \Gamma_\ast\colon\ \ \ A^i_{hom}(X)\ \xrightarrow{\cong}\ A^{i+4}_{hom}(Q)\ \ \ \forall i\ .\]
 In view of (\ref{h3}) and (\ref{h11}), this translates into
  \[   \Gamma_\ast\colon\ \ \ A^i_{}(h^3(X))\ \xrightarrow{\cong}\ A^{i+4}_{}(h^{11}(Q))=A^i(h^{11}(Q)(-4))\ \ \ \forall i\ .\]
  Using that the field $\C$ is a universal domain, this implies (cf. \cite[Lemma 1.1]{Huy}) there is an isomorphism of motives
  \[ \Gamma\colon\ \ \ h^3(X)\ \xrightarrow{\cong}\ h^{11}(Q)(-4)\ \ \  \hbox{in}\ \MM_{\rm rat}\ .\]
  The roles of $X$ and $Y$ being symmetric, the same argument also furnishes an isomorphism
  \[ \Psi\colon\ \ \ h^3(Y)\ \xrightarrow{\cong}\ h^{11}(Q)(-4)\ \ \  \hbox{in}\ \MM_{\rm rat}\ .\]
  The result is an isomorphism
  \[ \Psi^{-1}\circ \Gamma\colon\ \ \ h^3(X)\ \xrightarrow{\cong}\ h^3(Y)\ \ \  \hbox{in}\ \MM_{\rm rat}\ .\]
  Since the difference $h(X)-h^3(X)$ is just $\mathds{1} \oplus \mathds{1}(1)\oplus \mathds{1}(2) \oplus \mathds{1}(3)$, which is the same as $h(Y)-h^3(Y)$, there is also an isomorphism
  \[ h(X)\ \xrightarrow{\cong}\ h(Y)\ \ \ \hbox{in}\ \MM_{\rm rat}\ .\]

 \vskip0.5cm
 \noindent
 {\it Step 4: Spreading out.\/} We have now proven that there is an isomorphism of Chow motives $h^3(X)\cong h^3(Y)$ for a general pair of GPK$^3$ double mirrors. To extend this to {\em all\/} pairs of double mirrors, we reason as follows. Let
   \[ \pi_X\colon\ \ \ \XX\ \to\ B \]
   denote the universal family of all GPK$^3$ threefolds (so $B$ is an open in $PGL(\wedge^2 V)$). The double mirror construction corresponds to an involution $\sigma$ on $B$, such that $X_b$ and $Y_b:=X_{\sigma(b)}$ are double mirrors. We define the family $\YY$ as the composition
   \[   \pi_Y:=\sigma\circ\pi_X\colon\ \ \ \YY:=\XX\ \to\ B\ .\]
   The above construction of the correspondences $\Gamma,\Psi$ can be done relatively: over the open $B^0\subset B$ where the incidence variety $Q$ is smooth, one obtains relative correspondences
   \[ \Gamma\ \ \in A^3(\XX\times_{B^0}\YY)\ ,\ \ \Psi\ \ \in\ A^3(\YY\times_{B^0}\XX)\ ,\]
   such that for each $b\in B^0$ the restrictions $\Gamma_b\in A^3(X_b\times Y_b)$, $\Psi_b\in A^3(Y_b\times X_b)$ verify
   \[  \begin{split} &\pi^3_{X_b}= \Psi_b\circ \Gamma_b\ \ \ \hbox{in}\ A^3(X_b\times X_b)\ ,\\
                          &\pi^3_{Y_b}=  \Gamma_b\circ\Psi_b\ \ \ \hbox{in}\ A^3(Y_b\times Y_b)\ .\\
          \end{split}\]
      Taking the closure, one obtains extensions $\bar{\Gamma}\in A^3(\XX\times_B \YY)$, $\bar{\Psi}\in A^3(\YY\times_B \XX)$ to the larger base $B$, that restrict to $\Gamma$ resp. $\Psi$.    
          
     The correspondences $\pi^3_{X_b},\pi^3_{Y_b}$ also exist relatively (note that any GPK$^3$ threefold has Picard number $1$, and so $\pi^2,\pi^4$ exist as relative correspondences). The relative correspondences
     \[  \pi^3_\XX - \bar{\Psi}\circ\bar{\Gamma}\ \ \in\ A^3(\XX\times_B \XX)\ ,\ \ \  \pi^3_\YY - \bar{\Gamma}\circ\bar{\Psi}\ \ \in\ A^3(\YY\times_B \YY)   \]
     have the property that their restriction to a general fibre is rationally trivial. But this implies (cf. \cite[Lemma 3.2]{Vo}) that the restriction to {\em every\/} fibre is rationally trivial, and hence we obtain an isomorphism of motives for {\em all\/} $b\in B$, i.e. for all pairs $(X_b,Y_b)$ of double mirrors.                          
   \end{proof}

\begin{remark}
In the proof of theorem \ref{main}, we have not explicitly determined the inverse to the isomorphism $\Gamma$. With some more work, it is actually possible to show that there exists 
$m\in\QQ$ such that
  \[ {1\over m}\, {}^t \Gamma\colon\ \ h^{11}(Q)\ \to\ h^3(X)\ \ \ \hbox{in}\ \MM_{\rm rat} \]
  is inverse to $\Gamma$.
  \end{remark}

\begin{remark} It would be interesting to extend theorem \ref{main} to the category $\MM_{\Z rat}$ of Chow motives with integral coefficients. Let $X,Y$ be as in theorem \ref{main}. Is it true that $h(X)$ and $h(Y)$ are isomorphic in $\MM_{\Z{\rm rat}}$ ? 
Steps 1 and 2 in the above proof {\em probably\/} still work for Chow groups with $\Z$--coefficients (one just needs to upgrade the fibration results of section \ref{sfib} to $\Z$--coefficients); steps 3 and 4, however, certainly need $\QQ$--coefficients.
\end{remark}

\begin{remark} In all likelihood, an argument similar to that of theorem \ref{main} could also be applied to establish an isomorphism of Chow motives for the Grassmannian--Pfaffian Calabi--Yau varieties of \cite{Bor}, \cite{Mar}, as well as for the Calabi--Yau fivefolds of \cite{Man}.
\end{remark}

\section{Some corollaries}

\begin{corollary}\label{cor} Let $X,Y$ be two GPK$^3$ double mirrors. Let $M$ be any smooth projective variety. Then there are isomorphisms
  \[ N^j H^i(X\times M,\QQ)\cong N^j H^i(Y\times M,\QQ)\ \ \ \hbox{for\ all\ }i,j\ .\]
  (Here, $N^\ast$ denotes the coniveau filtration \cite{BO}.)
\end{corollary}

\begin{proof} Theorem \ref{main} implies there is an isomorphism of Chow motives $h(X\times M)\cong h(Y\times M)$. As the cohomology and the coniveau filtration only depend on the motive \cite{AK}, \cite[Proposition 1.2]{V1}, this proves the corollary.

\end{proof}

\begin{remark} It is worth noting that for any derived equivalent threefolds $X,Y$, there are isomorphisms
  \[ N^j H^i(X,\QQ)\cong N^j H^i(Y,\QQ)\ \ \ \hbox{for\ all\ }i,j\ ;\]
this is proven in \cite{ACV}.
 \end{remark}

\begin{corollary}\label{cor2} Let $X,Y$ be two GPK$^3$ double mirrors. Then there are (correspondence--induced) isomorphisms between higher Chow groups
  \[  A^i(X,j)_{}\ \xrightarrow{\cong}\ A^i(Y,j)_{} \ \ \ \hbox{for\ all\ }i,j\ .\]
  There are also (correspondence--induced) isomorphisms in higher algebraic K--theory
  \[ G_j(X)_{\QQ}\ \xrightarrow{\cong}\ G_j(Y)_{\QQ}\ \ \ \hbox{for\ all\ }j\ .\]
\end{corollary}

\begin{proof} This is immediate from the isomorphism of Chow motives $h(X)\cong h(Y)$.
\end{proof}

\vskip1cm
\begin{nonumberingt} Thanks to Len, for telling me wonderful stories of Baba Yaga and vatrouchka.

\end{nonumberingt}

\end{document}